\documentclass[12pt]{amsart}
\usepackage{amscd,amsmath,amsthm,amssymb,enumerate}
\usepackage[left]{lineno}
\usepackage{pstcol,pst-plot,pst-3d}
\usepackage{color}
\usepackage{pstricks}
\usepackage{stmaryrd}
\usepackage[utf8]{inputenc}
\usepackage{pstricks-add}
\usepackage[dvips]{graphicx}
\usepackage{epstopdf}
\usepackage{graphicx}
\newpsstyle{fatline}{linewidth=1.5pt}
\newpsstyle{fyp}{fillstyle=solid,fillcolor=verylight}
\definecolor{verylight}{gray}{0.97}
\definecolor{light}{gray}{0.9}
\definecolor{medium}{gray}{0.85}
\definecolor{dark}{gray}{0.6}

 \def\NZQ{\mathbb}               
 
 \def\QQ{{\NZQ Q}}

 \def\FF{{\NZQ F}}
 \def\GG{{\NZQ G}}

 \def\frk{\mathfrak}               

 \def\mm{{\frk m}}

 \def\G{{\mathcal G}}

 \def\opn#1#2{\def#1{\operatorname{#2}}} 
 \opn\chara{char} \opn\length{\ell} \opn\pd{pd} \opn\rk{rk}
 \opn\projdim{proj\,dim} \opn\injdim{inj\,dim} \opn\rank{rank}
 \opn\depth{depth} \opn\grade{grade} \opn\height{height}
 \opn\embdim{emb\,dim} \opn\codim{codim}
 
 \opn\Tr{Tr} \opn\bigrank{big\,rank}
 \opn\superheight{superheight}\opn\lcm{lcm}
 \opn\trdeg{tr\,deg}
 \opn\reg{reg} \opn\lreg{lreg} \opn\ini{in} \opn\lpd{lpd}
 \opn\size{size} \opn\sdepth{sdepth}
 \opn\link{link}\opn\fdepth{fdepth}\opn\lex{lex}
 \opn\tr{tr}
  \opn\Hilb{Hilb}
 \opn\type{type}
 \opn\gap{gap}
 \opn\arithdeg{arith-deg}
 \opn\div{div} \opn\Div{Div} \opn\cl{cl} \opn\Cl{Cl}
 \opn\Spec{Spec} \opn\Supp{Supp} \opn\supp{supp} \opn\Sing{Sing}
 \opn\Ass{Ass} \opn\Min{Min}\opn\Mon{Mon}
 \opn\Ann{Ann} \opn\Rad{Rad} \opn\Soc{Soc}
 \opn\Im{Im} \opn\Ker{Ker} \opn\Coker{Coker} \opn\Am{Am}
 \opn\Hom{Hom} \opn\Tor{Tor} \opn\Ext{Ext} \opn\End{End}
 \opn\Aut{Aut} \opn\id{id}
 
 \opn\nat{nat}
 \opn\pff{pf}
 \opn\Pf{Pf} \opn\GL{GL} \opn\SL{SL} \opn\mod{mod} \opn\ord{ord}
 \opn\Gin{Gin} \opn\Hilb{Hilb}\opn\sort{sort}
 \opn\PF{PF}\opn\Ap{Ap}
 \opn\mult{mult}
 \opn\bight{bight}
 \opn\aff{aff}
 \opn\relint{relint} \opn\st{st}
 \opn\lk{lk} \opn\cn{cn} \opn\core{core} \opn\vol{vol}  \opn\inp{inp} \opn\nilpot{nilpot}
 \opn\link{link} \opn\star{star}\opn\lex{lex}\opn\set{set}
 \opn\width{wd}
 \opn\Fr{F}
 \opn\QF{QF}
 \opn\G{G}
 \opn\type{type}\opn\res{res}
 \opn\conv{conv}
 \opn\Shad{Shad}
 \opn\ln{ln}
 \opn\gr{gr}
 
 \def\pot#1#2{#1[\kern-0.28ex[#2]\kern-0.28ex]}

 \opn\dirlim{\underrightarrow{\lim}}
 \opn\inivlim{\underleftarrow{\lim}}
 \let\tensor=\otimes
 \let\iso=\cong

 \let\to=\rightarrow
 
 \def\Implies{\ifmmode\Longrightarrow \else
         \unskip${}\Longrightarrow{}$\ignorespaces\fi}
 \def\implies{\ifmmode\Rightarrow \else
         \unskip${}\Rightarrow{}$\ignorespaces\fi}
 \def\iff{\ifmmode\Longleftrightarrow \else
         \unskip${}\Longleftrightarrow{}$\ignorespaces\fi}

 \let\:=\colon
 \newtheorem{Theorem}{Theorem}[section]
 \newtheorem{Lemma}[Theorem]{Lemma}
 \newtheorem{Corollary}[Theorem]{Corollary}
 \newtheorem{Proposition}[Theorem]{Proposition}

 \newtheorem{Example}[Theorem]{Example}

 \let\epsilon\varepsilon
 \let\kappa=\varkappa
 \def\qed{\ifhmode\textqed\fi
       \ifmmode\ifinner\quad\qedsymbol\else\dispqed\fi\fi}
 \def\textqed{\unskip\nobreak\penalty50
        \hskip2em\hbox{}\nobreak\hfil\qedsymbol
        \parfillskip=0pt \finalhyphendemerits=0}
 \def\dispqed{\rlap{\qquad\qedsymbol}}
 
 \opn\dis{dis}
 \def\pnt{{\raise0.5mm\hbox{\large\bf.}}}
 
 \opn\Lex{Lex}

\begin{document}

\title {On the initial behaviour of the number of generators  of powers of monomial ideals}

\author {Reza Abdolmaleki,  J\"urgen Herzog and Rashid Zaare-Nahandi}

\address{Reza Abdolmaleki, Department of Mathematics, Institute for Advanced Studies in Basic Sciences (IASBS), 45195-1159 Zanjan, Iran}
\email{abdolmaleki@iasbs.ac.ir}

\address{J\"urgen Herzog, Fachbereich Mathematik, Universit\"at Duisburg-Essen, Campus Essen, 45117
Essen, Germany}
\email{juergen.herzog@uni-essen.de}

\address{Rashid Zaare-Nahandi, Department of Mathematics, Institute for Advanced Studies in Basic Sciences (IASBS), 45195-1159 Zanjan, Iran}
\email{rashidzn@iasbs.ac.ir}

\dedicatory{ }

\begin{abstract} Given a number $q$,  we construct a  monomial ideal $I$  with the property that the function which describes the number of generators of $I^k$ has at least $q$ local maxima. 
\end{abstract}

\thanks{This paper was written while the first author was visiting Department of Mathematics of University Duisburg-Essen, Germany. He would like to thank  Professor Herzog for his support and hospitality}




\maketitle

\setcounter{tocdepth}{1}

\section*{Introduction}

In 1974 Judith Sally asked the second author whether there exists a one-dim\-en\-sion\-al local domain for which the square of the maximal ideal has less generators than the maximal ideal itself. In \cite{HW},  such an example has been provided. Later more such examples were found by other authors. On the other hand, in a polynomial ring the number of generators of the powers of  any non-principal ideal $I$  which is generated in a single degree is strictly increasing, and in the case of equigenerated monomial ideals there is general lower bound for the number of generators for each power  given by Freiman's theorem, see \cite{HMZ} and \cite{HHZ}.
Quite to the contrary,  if the monomial ideal $I$  is not generated in a single degree, then it may very well happen that $I^2$ has less generators than $I$. For monomial ideals $I$ in 2 variable,  a sharp lower bound for the number of generators $\mu(I^2)$ of $I^2$ is given in \cite{E1}, and Gasanova \cite{Ga} gave examples of monomial ideals $I$ with the property that for any given number $k$ one has $\mu(I^k)<\mu(I)$.

The question arises how ``wild'' the initial behaviour of the function $f_I(k)=\mu(I^k)$ could be for a monomial ideal in the polynomial ring. Of course for $k\gg 0$, $f_I(k)$ is a polynomial function since it is the Hilbert function of the fibre cone of $I$ (see \cite [Theorem 4.1.3]{B1}).  As a main result of this note we provide a family of monomial ideals $I$  with the property that the number of local maxima of $f_I$ exceeds any given number, see Theorem~\ref{main}. 

In Section~1 we introduce  the height $n$ monomial ideals $J\subset S=[x_1,\ldots,x_n]$,  $J=(x_1^{am},\ldots,x_n^{am})(x_1^m,\ldots,x_n^m)$, which we call the basic ideals of our construction (given by the parameters $a$ and $m$), and compute 
$\mu(J^k)$  and a socle degree $s(J^k)$ of $J^k$ for all $k \geq 1$, see Proposition~\ref{power} and Corollary~\ref{sp}.

Section~2 deals with the modified basic ideals $I$ which are obtained from $J$ by adding a $c$-th  power of the maximal ideal, where $c\leq s(J)$ and bigger than the  least degree 
of $J$. This ideal (with $n$ fixed) depends on the  parameters $a$, $c$ and $m$.  In the polynomial ring in $2$ variables we consider the modified basic ideal with parameters $a\geq 3$, $m \gg0$ and $c=s(J)-(a-2)m+1$, and denote this ideal by $I_{a,m}$.  For suitable choices of $a$ and $m$ it can be seen, that $f(k)=\mu(I_{a,m}^k)$ is strictly decreasing from a certain power on and for any given number of steps. The parameters can also be chosen that for any given number $b$, the function $f(k)$ has local maximum (respectively, local minimum) for $k=b$.

Finally in Section~3 we apply the results of the previous section to construct monomial ideals whose number of local maxima exceeds any given number.  These ideals are obtained as follows:  let $l\geq 1$ and choose for each $j=1,\ldots,l$ an ideal $I_{a_j,m_j}\in K[x_j,y_j]$, and let $I\subset K[x_1,\ldots,x_l,y_1,\ldots,y_l]$ be the ideal $I_{a_1,m_1}\cdots I_{a_l,m_l}$. The powers of such ideals tend to have many local maxima. More specifically, we let $a_j=ja$ and $m_j=a$ for all $j$. Then it is shown in Theorem~\ref{main}, that of a given integer $q$,  the integers $a$ and $l$ can be chosen that $I$ has at least $q$ local maxima. 

\section{The basic ideals of our construction and their  powers}

Let $K$ be a field and $S=K[x_1,\ldots,x_n]$ be the polynomial ring over $K$ in the variables $x_1,\ldots,x_n$. We denote by $\mm$ the unique maximal graded ideal of $S$, and by $\mu(I)$ the minimal number of generators of a graded ideal $I \subset S $.

We fix integers $ m \geq 1, a \geq  3 ,n \geq 2$, and set $ J=(x_{1}^{am}, \ldots , x_{n}^{am}) (x_{1}^{m}, \ldots , x_{n}^{m})$. This section is devoted to finding the socle degree of $J^k$, and the  number of minimal generators of $J^k$ for $k \geq 2$. These ideals allow us to construct ideals $I$ for which the number of generators  of the powers of $I$  have an unexpected behaviour.

\medskip
For the study of the ideal $J$ we need a few lemmata.  Let $H$ be a monomial ideal and $r\geq 1$ be an integer. The monomial ideal $H^{[r]}$ with $G(H^{[r]})=\{u^r \:\;  u \in G(H) \}$ is called the {\em pseudo-Frobenius power} of $H$.

\begin{Lemma}\label{resolution}
Let
\[
\FF: 0\to F_{p}\to F_{p-1}\to \cdots \to F_1\to F_0\to 0
\]
be a minimal graded free $S$-resolution of $S/H$ with $F_i=\bigoplus_jS(-a_{ij})$ for $i=1, \ldots ,p$. Then
\[
\GG: 0\to G_{p}\to G_{p-1}\to \cdots \to G_1\to G_0\to 0
\]
is a minimal graded free $S$-resolution of $S/H^{[r]}$ with $F_i=\bigoplus_jS(-a_{ij}r)$ for $i=1, \ldots ,p$.
\end{Lemma}
\begin{proof}
Consider the flat $K$-algebra homomorphism $\alpha: S\to S$ with $x_i\mapsto x_i^r$.  We view $S$ as an  $S$-module via $\alpha$ and denote it by $T$. Then $\GG=\FF\tensor_ST$. This yields the desired conclusion.
\end{proof}

Let $H \subset S$ be a graded ideal with $\dim(S/H)=0$. We denote by $s(H)$ the largest $i$ such that $(S/H)_i \neq 0$. This number is called the {\em socle degree} of $S/H$.

\begin{Lemma}\label{socle}
Let $H \subset S $ be an $\mm$-primary monomial ideal. Then
$$s(H^{[r]})=r(s(H)+n)-n.$$
\end{Lemma}
\begin{proof}
Let $\FF$ be a minimal graded free $S$-resolution of $H$ with $F_n=\bigoplus_jS(-a_{nj})$. Then $s(H)=a-n$, where $a=\max_j\{a_{nj}\}$,  because 
\[\Tor_n(K, S/H)\iso H_n(x_1,\ldots,x_n;S/H)\]
 as graded $K$-vector spaces, and since $H_n(x_1,\ldots,x_n;S/H)$ is generated by the elements $ue_1 \wedge \cdots \wedge e_n$ with $u \in (H:\mm )/H$. Applying Lemma~\ref{resolution}, we have $s(H^{[r]})=ra-n=r(s(H)+n)-n$.
\end{proof}

Now let $E=(x_{1}^{a}, \ldots , x_{n}^{a}) (x_{1}, \ldots , x_{n})$. Then $J^k=(E^k)^{[m]}$.

\begin{Proposition}
\label{soclepower}
With the notation introduced we have
\[
s(E^k)= \begin{cases}
a(k-1)+(a-1)n, & if\hspace{3mm} k \leq (n-1)(a-1);\\
(a+1)k-1, & if\hspace{3mm}k \geq (n-1)(a-1).
\end{cases}
\]
\end{Proposition}

\begin{proof}
Let  $A=(x_{1}^{a}, \ldots , x_{n}^{a})$. Then, $E^k=A^k\mm^k=(A^k)_{\geq ak+k}$. Hence if $s(A^k)> ak+k$,  then $s(E^k)=s(A^k)$. Because $A^k=(\mm^k)^{[a]}$, we may apply Lemma~\ref{socle} and get  $s(A^k)=a(k-1+n)-n=a(k-1)+(a-1)n$. Since $ a(k-1)+(a-1)n =s(A^k)> ak+k$ if and only if  $k \leq (n-1)(a-1)$, we obtain the desired result for $k \leq (n-1)(a-1)$.

Now suppose  that $k > (n-1)(a-1)$. Then $s(A^k)< ak+k$. Therefore, $E^k=(A^k)_{\geq ak+k}=\mm^{ak+k}$. Since $s(\mm^{ak+k})=ak+k-1$, the  proof is completed.
\end{proof}

\begin{Corollary}\label{sp}
With the notation introduced we have
$$
s(J^k)= \begin{cases}
ma(k+n-1)-n, & if\hspace{3mm} k \leq (n-1)(a-1);\\
m(ak+k+n-1)-n, & if\hspace{3mm}k \geq(n-1)(a-1).
\end{cases}
$$
\end{Corollary}
\begin{proof}
Since $J^k=(E^k)^{[m]}$, the assertion follows from Lemma~\ref{socle} and Proposition~\ref{soclepower}.
\end{proof}

In the following statement we only consider the case $n=2$. 

\begin{Proposition}
\label{power}
Let $J=(x_{1}^{am},x_2^{am})(x_{1}^{m},x_{2}^{m}) \subset K[x_{1},x_{2}]$. Then
\[
\mu(J^k)=\mu(E^k)= \begin{cases}
(k+1)^2, & if\hspace{3mm} k \leq a-1;\\
(a+1)k+1, & if\hspace{3mm}k \geq a-1.
\end{cases}
\]
\end{Proposition}
\begin{proof}
Since $J^k$ is a Frobenius power of $E^k$, we have $\mu(J^k)=\mu(E^k)$. Let $A=(x_1^a, x_2^a)$. Then $E^k=A^k\mm^k$, and  hence $E^k$ is generated by the set of monomials $\mathcal{S}=\{ uv \: u \in G(A^k), v \in G(\mm^k) \}$. Suppose that for monomials $u,u' \in G(A^k) $ and $v,v' \in G(\mm^k)$ we have $uv=u'v'$.  Let $ u=x_{1}^{ia}x_{2}^{(k-i)a}$ and $u'=x_{1}^{ja}x_{2}^{(k-j)a}$ for some $i,j \in \{0, \ldots, k \}$. Also let   $ v=x_{1}^{r}x_{2}^{k-r}$
and $ v'=x_{1}^{s}x_{2}^{k-s}$  for some $r,s \in \{0, \ldots, k \}$.  We may assume that $i\geq j$. Then $s\geq r$, and moreover  $ia+r=ja+s$. Therefore, $(i-j)a=s-r \leq k$. Hence if $k<a$, then  $r=s$ and $i=j$. This implies that  
\[
\mu(E^k)=|\mathcal{S}|=\mu(A^k)\mu(m^k)=(k+1)(k+1)=(k+1)^2.
\]

Now, let $k > a-1$.  By Proposition~\ref{soclepower}, we have $s(E^k) = (a+1)k-1$. Since $E^k$ is generated in degree $ka+k> s(E^k)$, it follows that  $E^k=\mm^{ak+k}$. Hence,
$\mu(E^k)=\mu(\mm^{ak+k})=ak+k+1$.
\end{proof}

\section{On the powers of the modified basic ideals}
In this section we let $a \geq 3$ and $m \geq 1$ be integers, and set
\begin{eqnarray}
\label{professor}
 J=(x_{1}^{am}, \ldots , x_{n}^{am}) (x_{1}^{m}, \ldots , x_{n}^{m}) \quad \text{and} \quad I=J+\mm^c,
\end{eqnarray}
where $c> d=(a+1)m$. Note that $J$  is generated in  degree $d$. 

\begin{Proposition}
\label{necessary}
Consider the ideal $J$ as we introduced above, and suppose that $I=J+\mm^{c}$  where $c>d$. Let $k\leq (n-1)(a-1)$. Then, $I^k=J^k$, if  $c\geq s(J)-(k-1)m+1$.
\end{Proposition}
\begin{proof}
We have  $I^k=J^k$, if  $J^{k-1}\mm^{c}+ \ldots +J(\mm^{c})^{k-1}+ (\mm^{c})^k \subset J^k$. Since $J^{k-1}\mm^{c}$ is the summand with the least degree, then we have $I^k=J^k$, if $(k-1)d+c>s(J^k)$. Using the first equality in Proposition~\ref{sp}, we have $c\geq s(J)-(k-1)m+1$.
\end{proof}

\begin{Proposition}
\label{second}
Let ideals $J$ and $I$ be as in \eqref{professor}. 
\begin{enumerate}[{\rm(a)}]
\item If  $c\geq d+(n-1)(m-1)$, then $I^k=J^k$ for all $k\geq (n-1)(a-1)$.

\item  $I^k=J^k$ for some $k$ if and only if  $c\geq d+(n-1)(m-1)$.\end{enumerate}

\end{Proposition}
\begin{proof}
(a) Since $c\geq d+(n-1)(m-1)$, it follows from Proposition~\ref{sp} that $(k-1)d+c>s(J^k)$ for $k\geq (n-1)(a-1)$. Therefore, $I^k=J^k$ for $k\geq (n-1)(a-1)$.

 (b) Because of (a) we only need to show that $c\geq d+(n-1)(m-1)$ if $I^k=J^k$ for some $k$. Assume $c<d+(n-1)(m-1)$. Then, we show that $I^k \neq J^k$ for all $k$, which will be a contradiction. It is enough to prove $I^k \neq J^k$ for all $k$ when $c=d+(n-1)(m-1)-1$. Let $u=x_{1}^{kd-1} x_{2}^{m-1}\ldots x_{n}^{m-1}$. Then we have
 \[
 u =(x_{1}^{(k-1)d})(x_{1}^{d-1} x_{2}^{m-1}\ldots x_{n}^{m-1})\in J^{k-1}\mm^{c},
\]
   but $u \notin J^k$ because, any monomial in $J^k$, is generated by an element of the form $v=(x_{1}^{k_1am} \ldots x_{n}^{k_nam})(x_{1}^{j_1m} \ldots x_{n}^{j_nm})$ with $k_1+\cdots +k_n=k$ and $j_1+\cdots +j_n=k$.    Therefore, $J^{k-1}\mm^{c} \nsubseteq J^k$, and so $I^k \neq J^k$.
\end{proof}

By Proposition~\ref{second}, if $I^k=J^k$ for some $k$, then $I^k=J^k$ for all $k>(n-1)(a-1)$. The next result gives a more precise statement about the smallest number $k$ for which $I^k=J^k$.
\begin{Proposition}
\label{equality}
Let $ J$ and $I$ be as in \eqref{professor}. Then, for $2 \leq k \leq (n-1)(a-1)$ we have  $I^k=J^k$ and $I^{k-1}\neq  J^{k-1}$, if and only if
\begin{eqnarray}
\label{interval}
s(J) -(k-1)m+1 \leq c \leq s(J) -(k-2)m.
\end{eqnarray}
\end{Proposition}
\begin{proof}
The first inequality is proven in Proposition~\ref{necessary}.
For the second inequality we have $I^{k-1}\neq J^{k-1}$ if and only if  $J^{k-2}\mm^c+ \ldots +J(\mm^c)^{k-2}+ (\mm^c)^{k-1} \not \subset J^{k-1}$. Then to have $I^{k-1}\neq J^{k-1}$, we must have
\begin{align*}
(k-2)d+c \leq s(J^{k-1}).
\end{align*}
Using the first equality in Proposition~\ref{sp}, we have $c\leq s(J)-(k-2)m$.
\end{proof}

\begin{Lemma}
\label{need}
Let $S=K[x_1, \ldots , x_n]$  and  $a , k,c\geq 1$ be integers. Then,
\[
(x_{1}^{a}, \ldots , x_{n}^{a})^k\mm^c=\mm^{ka+c}
\]
if and only if $c \geq (n-1)(a-1)$.
\end{Lemma}
\begin{proof}
Let $F=(x_{1}^{a}, \ldots , x_{n}^{a})^k$. Then $(x_{1}^{a}, \ldots , x_{n}^{a})^k\mm^c=\mm^{ka+c}$ if and only if $F_{\geq ak+c}=\mm^{ka+c}$, which is the case if and only if $ak+c\geq s(F)+1$. The desired conclusion follows from Lemma~\ref{socle}.
\end{proof}

\begin{Lemma}
\label{difference}
Let $I$ and $J$ be as in \eqref{professor}  and $1 \leq k \leq (n-1)(a-1)$.  Then, for $c \geq s(J) -(k-1)m+1$ we have
\begin{enumerate}[{\rm(a)}]
\item $(x_{1}^{m}, \ldots , x_{n}^{m})^k\mm^{c-d}=\mm^{km+c-d}$.

\item $J^k\mm^{c-d}=(x_{1}^{am}, \ldots , x_{n}^{am})^k\mm^{km+c-d}$.

\item $J^{k-1}\mm^c=\mm^{(k-1)d+c}$.

\item $I^k=J^k+J^{k-1}\mm^{c}$.

\end{enumerate}
\end{Lemma}

\begin{proof}

(a) Since $c-d \geq (n-1)(m-1)$ by the assumption, the proof follows from Lemma~\ref{need}.

(b)follows from (a).

(c) Note that $J^{k-1}\mm^{c}= \mm^{(k-1)d+c}$ if and only if $(k-1)d+c >s(J^{k-1})$. Thus the statement follows from  Corollary~\ref{sp} and the assumption $c \geq s(J) -(k-1)m+1$.

(d) We have $I^k=J^k+J^{k-1}\mm^c+ \ldots +J(\mm^c)^{k-1}+(\mm^c)^k $.  On the other hand, we have $J^{k-1}\mm^c=\mm^{(k-1)d+c}$ by (c). Hence, $I^k=J^k+J^{k-1}\mm^c$ because the other summands are contained in $\mm^{(k-1)d+c}$, since they are generated in degree greater than $(k-1)d+c$.
\end{proof}
 The ideals as in \eqref{professor} depend on three parameters: $a$, $m$ and $c$. In the following we fix $n=2$ and let $c=s(J)-(a-2)m+1=d+m-1=(a+2)m-1$. For these choices of $n$ and $c$ we denote the ideal defined in \eqref{professor} by $I_{a,m}$. In other words, $I_{a,m}=(x_{1}^{am},x_2^{am})(x_{1}^{m},x_{2}^{m})+(x_1,x_2)^{(a+2)m-1}$.

\begin{Proposition}
\label{total}
Let as before $a \geq 3$ and $m \geq 1$. Then
\[
\mu(I_{a,m}^k)= \begin{cases}
(1-m)k^2+(am-m+2)k+1, & if\hspace{3mm} k \leq a-1;\\
(a+1)k+1, & if\hspace{3mm}k \geq a-1.
\end{cases}
\]
\end{Proposition}
\begin{proof}
Note that the two formulas for $\mu(I_{a,m}^k)$ coincide for $k=a-1$.
Now let $k < a-1$. By Proposition~\ref{equality}, $I_{a,m}^k\neq J^k$ where $J=(x_{1}^{am},x_2^{am})(x_{1}^{m},x_{2}^{m})$.

 First we show that in this case
\begin{eqnarray*}
 \mu(I_{a,m}^k)-\mu(J^k)= -mk^2+(am-m)k.
\end{eqnarray*}
Indeed, it is clear that $J^k\mm^{m-1} \subset \mm^{(ka+k+1)m-1}$. Then, by using Lemma~\ref{difference} (d) and (c), we get
\begin{align}
\label{minus}
\mu(I_{a,m}^{k})-\mu(J^k)= \mu(\mm^{(ka+k+1)m-1})-\mu(J^k\mm^{m-1}).
\end{align}
By Lemma~\ref{difference} (b), we have $J^k\mm^{m-1}=(x_1^{am},x_2^{am})^k\mm^{(k+1)m-1}$. Thus the set
\[\{ x_1^{(k-j)am+i}x_2^{jam+(k+1)m-1-i} \:\;  i=0, \dots , (k+1)m-1 \text{ and } j=0, \ldots , k \}
\]
generates $J^k\mm^{m-1}$.
Since $k < a-1$, it follows that $am > (k+1)m-1$. Therefore,
\begin{eqnarray*}
&&0<1< \cdots <am<am+1<\cdots<am+(k+1)m-1<2am\\&<&\cdots <(k-1)am+(k+1)m-1<ram<\cdots<ram+(k+1)m-1.
\end{eqnarray*}
This shows that $\mu(J^k\mm^{m-1})=(k+1)(k+1)m=(k+1)^2m$. Therefore \eqref{minus} implies that
\[
 \mu(I_{a,m}^k)-\mu(J^k)=(ka+k+1)m-(k+1)^2m=-mk^2+(am-m)k.
\]
By Proposition~\ref{power}, $\mu(J^k)=(k+1)^2$ for $k < a-1$.  Thus in this case,
\[
\mu(I_{a,m}^k)=(k+1)^2-mk^2+(am-m)k=(1-m)k^2+(am-m+2)k+1. 
\]

 For $k \geq a-1$, Proposition~\ref{second} gives us that $I_{a,m}^k=J^k$. So, by Proposition~\ref{power}, $\mu(I_{a,m}^k)=(a+1)k+1$. 
\end{proof}
For any real number $\alpha$, the round of $\alpha$, which is denoted by $\lfloor \alpha \rceil $ is defined as the nearest integer to $\alpha$, that is, 
\[
\lfloor \alpha \rceil = \begin{cases}
\lfloor \alpha \rfloor, & if\hspace{3mm}  \alpha - \lfloor \alpha \rfloor < 1/2;\\
\lceil \alpha \rceil, & if\hspace{3mm}  \alpha - \lfloor \alpha \rfloor \geq 1/2,
\end{cases}
\]
where $\lfloor \alpha \rfloor $ is the largest integer less than or equal to $ \alpha $, and  $\lceil \alpha \rceil$ is the smallest integer greater than or equal to $\alpha$.

\begin{Corollary}
\label{max}
Let  $a\geq 3$ and $ m \gg 0$, and let  $t=\lfloor \dfrac{(a-1)m+2}{2m-2}\rceil$.  Then $\mu(I_{a,m}^t) > \mu(I_{a,m}^{k})$ for all $k \in \{ 1, \ldots , a-1 \}$, $k\neq t$. 
For $m\gg 0$, $t \approx \frac{a-1}{2}$. 
\end{Corollary}
\begin{proof}
Set $f(x)=(1-m)x^2+(am-m+2)x+1$. Then $f(k)=\mu(I_{a,m}^k)$ for all integers $1 \leq k \leq a-1$. Taking the first derivative of $f(x)$, we get $f'(x)=2(1-m)x+(am-m+2)$. Therefore, $f'(s)=0$ for $s= \dfrac{(a-1)m+2}{2m-2}$. Since the coefficient of $x^2$ in the equation of $f(x)$ is negative, the point $p=(s,f(s) )$ is a local maximum of $f(x)$. By the symmetry of the parabola, it follows that 
$f(k)=\mu(I_{a,m}^k)$ has a  local maximum, if and only if there exists no integer $h$ such that $s=h+1/2$. This is the case if and only if  $\dfrac{(a-1)m+2}{m-1}$ is not an odd integer. For $a\ge 3$ and $m \gg0$  this fraction is not even an integer.
\end{proof}

\begin{Corollary}
For any integer $k_0>1$, there exist integers $a$ and $m$  such that the function $f(k)=\mu(I_{a,m}^k)$ has a local maximum for $k=k_0$.
 \end{Corollary}
\begin{proof}
Let $\varepsilon=\dfrac{(a-1)m+2}{2m-2}-\dfrac{a-1}{2}$. Then $0<\varepsilon<1/2$ for $m \gg 0$. Let $a$ be an odd integer. Then $\dfrac{a-1}{2}$ is an integer. Therefore, $\lfloor \dfrac{(a-1)m+2}{2m-2}\rceil=\dfrac{a-1}{2}$ for $m\gg0$.  The desired conclusion follows from Corollary~\ref{max} by choosing $m\gg0$ and $a =2k_0+1$.
\end{proof}
\begin{Corollary}
For any integer $k_0>1$, there exists integers $a$ and $m$,    such that the function $f(k)=\mu(I_{a,m}^k)$ has a local minimum for $k=k_0$.
\end{Corollary}
\begin{proof}
The ideal $I_{a,m}$ has the desired property when $a=k_1+1$ and $m\gg 0$.
\end{proof}

 Similarly we conclude   
   
\begin{Corollary}
\label{down}
Given an integer $b>0$, there exist integers $b_0$, $a$ and $m$  such that $\mu(I_{a,m}^k)>\mu(I_{a,m}^{k+1})$ for all $k$ with $b_0\leq k\leq b_0+b-1$, and $\mu(I_{a,m}^{k})<\mu(I_{a,m}^{k+1})$ for $k\geq b_0+b-1$.
\end{Corollary} 

\section{Families of ideals $I$ for which the local maxima of $f(k)=\mu(I^k)$ exceeds any given number}
 In this section we use ideals $I_{a,m}$ and their products to obtain ideals $I$, such that the function $ f(k)=\mu(I^k)$ has local maxima as many as we want.  

\medskip
For any  given integer $q$, we construct a monomial  ideal $I$ of height 2  with the following property: for $i=1,\ldots,q$ there exist integers $s_i <r_i<t_i$ such that the intervals $[s_i,t_i]$ are  pairwise disjoint intervals  and $f(s_i)<f(r_i)>f(t_i)$ for $i=1,\ldots,q$. In other words, there exist monomial ideals $I$ for which   $f(r)=\mu(I^r)$  has  at least $q$  local maxima.

 Let $I_{a_1,m_1} \subset K[x_1,y_1], \ldots , I_{a_l,m_l} \subset K[x_l,y_l]$ be ideals as defined before,  but in polynomial rings in pairwise disjoint sets of variables.  Set 
\begin{eqnarray*}
\label{ideal}
I=I_{a_1,m_1}I_{a_2,m_2} \ldots I_{a_l,m_l}. 
\end{eqnarray*}
For any $j \in \{ 1, \ldots , l \}$, by Proposition~\ref{total} we have
\[
\mu(I_{a_j,m_j}^k)=\begin{cases}
(1-m_j)k^2+(a_jm_j-m+2)k+1, & if\hspace{3mm} k \leq a_j-1;\\
a_jk+k+1,   & if\hspace{3mm} k \geq a_j.
\end{cases}
\]
We now choose particular values for  the  integers $a_j$ and $m_j$,  depending on an integer $a\geq 3$. We choose $m_j=a$  for all $j$ and $a_j=ja$ for all $j$. Then

\begin{eqnarray}
\label{ideal}
I=I_{a,a}I_{2a,a} \ldots I_{la,a} , \text{ and}  
\end{eqnarray}
\[
\mu(I_{ja,m}^k)=\begin{cases}
(1-a)k^2+(ja^2-a+2)k+1, & if\hspace{3mm} k \leq ja-1;\\
jak+k+1,  & if\hspace{3mm} k \geq ja.
\end{cases}
\]

Hence, 
\[
\mu(I^k)=\begin{cases}
f_1f_2 \ldots f_l, & if\hspace{3mm} 1 \leq k \leq a-1;\\
g_1f_2 \ldots f_l, & if\hspace{3mm} a \leq k \leq 2a-1;\\
\vdots\\
g_1 \ldots g_{l-1}f_l,  & if\hspace{3mm} (l-1)a \leq k \leq la-1;\\
g_1 \ldots g_l, & if\hspace{3mm} k \geq la.
\end{cases}
\]
where $f_j=(1-a)k^2+(ja^2-a+2)k+1$ and $g_j=jak+k+1 $ for $j\in \{1,\ldots ,l\}$.

\begin{Theorem}
\label{main}
Let $q\geq 1$ be an integer. Then there exist integers $a$ and $l$ such that for the ideal $I$  defined in  \eqref{ideal}, the function $f(k)=\mu(I^k)$ has at least $q$ local maxima. 
\end{Theorem}
\begin{proof}
Let $p \in \QQ[x]$. We denote by $L_1(p)$ the leading of $p$ and by $L_2(p)$ the leading term of $p-L_1(p)$.
We show that
\begin{enumerate}[{\rm(a)}]
\item$f(ia)>f((i+1)a-1)$ for $i \in \{ 1, \ldots ,l-1\}$ and $a \gg 0$.
\item $f(ia+1)>f(ia)$ for $i \in \{ 1, \ldots ,l-1\}$ and $a \gg 0$ , if $l/i>1+1/2+1/3+ \cdots +1/(l-i)$.
\item Let $i \leq \sqrt{l} $ and $l/i>1+1/2+1/3+ \cdots +1/(l-i)$. Then 
\[
l/j>1+1/2+1/3+ \cdots +1/(l-j) \quad \text{for all} \quad j <i.
 \]
\item  Let $\lambda \geq 4$ be an integer, $l=\lambda^2$, and $i \in \{ 1, \ldots , \lambda \}$. Then $l/i>1+1/2+1/3+ \cdots +1/(l-i)$.
\end{enumerate}
Proof of (a): We compare  $L_1(f(ia))$ and $L_1(f((i+1)a-1))$  for $a \gg0$. For $k \in [ia,(i+1)a-1]$ we have
\[
f(k)=\prod_{j=1}^{i}((ja+1)k+1)\prod_{j=i+1}^{l}((1-a)k^2+(ja^2-a+2)k+1).
\]
Therefore,
  {\footnotesize
\begin{eqnarray*}
&&f((i+1)a-1)=\\ && \prod_{j=1}^{i}((ja+1)((i+1)a-1)+1)\prod_{j=i+1}^{l}((1-a)((i+1)a-1)^2+(ja^2-a+2)((i+1)a-1)+1).
\end{eqnarray*}}
Note that, in the first product all factors are  polynomials in $a$ of degree $2$, and in the second product the first factor is a polynomial in $a$ of degree $2$ and others are  polynomials in $a$ of degree $3$. Therefore, $L_1(f((i+1)a-1))$ is a polynomial in $a$ of degree $3l-i-1$. On the other hand,
\begin{eqnarray*}
f(ia)=\prod_{j=1}^{i}((ja+1)(ia)+1)\prod_{j=i+1}^{l}((1-a)(ia)^2+(ja^2-a+2)(ia)+1).
\end{eqnarray*}
It follows that $L_1(f(ia)$ is a polynomial in $a$ of degree degree $3l-i$. So $f(ia)>f((i+1)a-1)$ for $a\gg 0$.

Proof of (b): We have
 {\small
\begin{eqnarray*}
f(ia+1)=\prod_{j=1}^{i}((ja+1)((ia+1)+1)\prod_{j=i+1}^{l}((1-a)(ia+1)^2+(ja^2-a+2)(ia+1)+1).
\end{eqnarray*}}

 Note that $L_1(f(ia+1))=L_1(f(ia))$. Thus we must compute and compare $L_2(f(ia+1))$ and $L_2(f(ia))$. In general, if  $h_1, \ldots, h_r \in \QQ[x]$, then 
\begin{eqnarray}
\label{prod}
 L_2(h_1 \cdots h_r)=\sum_{s=1}^{r}L_2(h_s)\prod_{t=1 \atop t \neq s}^{r}L_1(h_t).
\end{eqnarray}

By using \eqref{prod} we get
\begin{eqnarray*}
 L_2(f(ia))=(l-i)!i^{l}a^{3l-i-1}(\sum_{j=1}^{i}i!/j) + i!i^{l-1}  a^{3l-i-1}(\sum_{j=1}^{l-i}(i^2-i)(l-i)!/j).
\end{eqnarray*}

Similarly we get
\begin{eqnarray*}
L_2(f(ia+1))  &=&(l-i)!i^{l}a^{3l-i-1}(\sum_{j=1}^{i}(i+j)!/j) \\ &&+ i!i^{l-1}a^{3l-i-1}(\sum_{j=1}^{l-i}(i^2-2i+j)(l-i)!/j)\\&=&(l-i)!i^{l}a^{3l-i-1}(\sum_{j=1}^{i}i!/j)+ i!(l-i)!i^{l}  a^{3l-i-1} \\ &&+ i!i^{l-1}a^{3l-i-1}(\sum_{j=1}^{l-i}(i^2-2i+j)(l-i)!/j).
 \end{eqnarray*}
Hence
\begin{eqnarray*}
 L_2(f(ia+1))-L_2(f(ia))=i!(l-i)!i^{l-1}a^{3l-i-1}(\sum_{j=1}^{l-i}(-i+j)/j)+i!(l-i)!i^{l}a^{3l-i-1}.
\end{eqnarray*}
So $f(ia+1)>f(ia)$ for $a\gg 0$, if $L_2(f(ia+1))-L_2(f(ia))>0$, which is equivalent to have 
 \begin{eqnarray*}
 i!(l-i)!i^{l-1}a^{3l-i-1}(\sum_{j=1}^{l-i}(-i+j)/j+i)>0.
\end{eqnarray*}
Therefore, if 
\begin{eqnarray*}
\sum_{j=1}^{l-i}(-i+j)/j+i>0, 
\end{eqnarray*} 
then, $L_2(f(ia+1))-L_2(f(ia))>0$. So, it is enough to have
\begin{eqnarray*}
l-i>i/2+i/3+ \cdots +i/(l-i), 
\end{eqnarray*}
which is equivalent to have $l/i>1+1/2+1/3+ \cdots +1/(l-i)$.

Proof of (c). By assumption $l \geq i^2$. Therefore,  $l+1>i^2$, and hence $l-i+1>i^2-i$. It follows that  $1/(l-i+1)<1/(i^2-i)<l/(i^2-i)=l/(i-1)-l/i$.
Thus 
\[
l/(i-1)>l/i+1/(l-i+1).
\]
So, since $l/i > 1+1/2+1/3+ \cdots +1/(l-i)$ by the assumption, we get
\[
l/(i-1)>1+1/2+1/3+ \cdots +1/(l-i)+1/(l-i+1).
\]
By using induction on $i-j$ ($i$ is fixed) desired conclusion follows.

Proof of (d). Let $\lambda \geq 4$ be an integer, $l=\lambda^2$, and $i \in \{ 4, \ldots , \lambda \}$. It is well known that $\sum\limits_{r=1}^{\alpha}1/r<\ln(\alpha)+1$ for any positive integer $\alpha$. Therefore,
\begin{eqnarray*}
1+1/2+1/3+ \cdots +1/(\lambda^2-\lambda)&<&\ln (\lambda^2-\lambda)+1=\ln (\lambda(\lambda-1))+1\\&=&\ln (\lambda)+\ln (\lambda-1)+1<2\ln(\lambda)+1.
\end{eqnarray*}
 Observe that $i>2\ln(i)+1$ for $i\geq 4$. Therefore, 
\begin{eqnarray*}
l/\lambda =\lambda>2\ln(\lambda)+1 &>& 1+1/2+1/3+ \cdots +1/(\lambda^2-\lambda)\\&=&1+1/2+1/3+ \cdots +1/(l-\lambda).
\end{eqnarray*}
So, by (c) we get 
\[
l/i  > 1+1/2+1/3+ \cdots +1/(l-i) \quad \text{for} \quad i=1, \ldots , \lambda.
\]

 Now, given a positive integer $q$, we construct an ideal with at least $q$ local maxima: we choose $l =(q+1)^2$. So $l \geq i^2$ for  $i=1, \ldots, q+1$. Therefore, it follows from  (d) and (c) that 
 \[
 l/i>1+1/2+1/3+ \cdots +1/(l-i) \quad  \text{for} \quad i=1, \ldots q+1.
 \]
 Hence, from (b) we get $f(ia+1)>f(ia)$ for $i=1, \ldots q+1$ and $a \gg 0$. This together with (a)  completes the proof.
  \end{proof}
  
\begin{Example}{\em 
Let $l=16$ and $a=20$. So $I=I_{20,20}I_{40,20} \cdots I_{320,20}$.  Then Theorem~\ref{main} implies that we have least $3$ local maxima. Indeed, 
$\mu(I^{37})<\mu(I^{38})>\mu(I^{39})$, $\mu(I^{55})<\mu(I^{56})>\mu(I^{57})$ and $\mu(I^{72})<\mu(I^{73})>\mu(I^{74})$. Computationally it can be checked that $\mu(I^k)$ has actually  $8$ local maxima. }
\end{Example}

\end{document}